\documentclass[final,leqno,onefignum,onetabnum]{siamltex}
\usepackage{amsmath}
\usepackage{amssymb}
\usepackage{graphicx}
\usepackage{algorithm}
\usepackage{color}

\newtheorem{example}[theorem]{\protect\examplename}
\providecommand{\examplename}{Example}

\providecommand{\remarkname}{Remark}
\newtheorem{assumption}[theorem]{\protect\assumptionname}
\providecommand{\assumptionname}{Assumption}

\newcommand{\ee}{\text{e}}
\newcommand{\dd}{\text{d}}
\newcommand{\wh}{\widehat}
\newcommand{\wt}{\widetilde}
\newcommand{\tcone}{\raisebox{.5pt}{\textcircled{\raisebox{-.9pt} {\small 1}}}}
\newcommand{\tctwo}{\raisebox{.5pt}{\textcircled{\raisebox{-.9pt} {\small 2}}}}
\newcommand{\tcthree}{\raisebox{.5pt}{\textcircled{\raisebox{-.9pt} {\small 3}}}}
\newcommand{\msp}{\rule[0mm]{0mm}{10.5pt}}

\title{Overcoming order reduction in diffusion-reaction splitting. Part 2: oblique boundary conditions
}

\newcommand*\samethanks[1][\value{footnote}]{\footnotemark[#1]}
\author{
Lukas Einkemmer\thanks{Department of Mathematics, University of Innsbruck, Technikerstra\ss e 13, Innsbruck, Austria ({\tt lukas.einkemmer@uibk.ac.at}, {\tt alexander.ostermann@uibk.ac.at}).
} \and Alexander Ostermann\samethanks
}

\begin{document}

\maketitle

\slugger{sisc}{xxxx}{xx}{x}{x--x}

\begin{abstract}
Splitting methods constitute a well-established class of numerical schemes for the time integration of partial differential equations. Their main advantages over more traditional schemes are computational efficiency and superior geometric properties. In the presence of non-trivial boundary conditions, however, splitting methods usually suffer from order reduction and some additional loss of accuracy. For diffusion-reaction equations with inhomogeneous oblique boundary conditions, a modification of the classic second-order Strang splitting is proposed that successfully resolves the problem of order reduction. The same correction also improves the accuracy of the classic first-order Lie splitting. The proposed modification only depends on the available boundary data and, in the case of non Dirichlet boundary conditions, on the computed numerical solution. Consequently, this modification can be implemented in an efficient way, which makes the modified splitting schemes superior to their classic versions. The framework employed in our error analysis also allows us to explain the fractional orders of convergence that are often encountered for classic Strang splitting. Numerical experiments that illustrate the theory are provided.
\end{abstract}



\pagestyle{myheadings}
\thispagestyle{plain}
\markboth{L.~EINKEMMER, A.~OSTERMANN}{OVERCOMING ORDER REDUCTION IN SPLITTING METHODS II}

\section{Introduction}

For the numerical solution of evolution equations, splitting methods have attracted much attention in recent years. The main advantage of splitting methods over more traditional time integration schemes is the fact that the partial vector fields in the splitting can usually be integrated more efficiently, sometimes much more efficiently, compared to applying a numerical method to the full problem. Moreover, splitting methods have, in general, better geometric properties \cite{hairer2006}. As a typical example, we mention splitting of the Schr\"odinger equation \cite{faou12} into the kinetic and the potential part. The former is integrated in frequency space using the fast Fourier transform, whereas the latter can often be integrated exactly. This requires considerably less computational effort compared to directly applying, for example, an implicit Runge--Kutta method to the Schr\"odinger equation.

In this paper, we will consider a splitting approach for diffusion-reaction equations. In our problem, the diffusion is modelled by a linear elliptic differential operator and the reaction by a nonlinear smooth function. Such a splitting into a linear elliptic problem, which can be solved efficiently by fast Poisson techniques \cite{mckenney1995fast}, and a nonlinear ordinary differential equation, which is simply solved pointwise in space, has attracted much attention in the literature (see, e.g., the articles \cite{DW91,descombes01, estep08, gerisch2002, hansen2012} and the monograph \cite{hundsdorfer2003numerical}).

Although splitting methods are widely used in multiphysics problems, a rigorous error analysis is still missing for many important applications. For instance, the convergence analysis of splitting methods is often carried out either for periodic boundary conditions or for full space problems. On the other hand, it is known from numerical experiments that non-trivial boundary conditions lead to order reduction (see, e.g., the experiments in \cite{EO15overcoming, hansen2009}).

In order to avoid such an order reduction, the splitting must be modified in such a way that the internal steps of the splitting become compatible with the prescribed boundary conditions. For reaction-diffusion equations with Dirichlet boundary conditions, we presented such a correction in our recent paper \cite{EO15overcoming}. Here, we propose a modification of that idea which, on the one hand, is easier to implement, and on the other hand can be extended to more general oblique boundary conditions. The latter comprise Dirchlet, Neumann, Robin, and mixed boundary conditions.

To illustrate the underlying ideas of our approach, we consider for a moment the semilinear heat equation
\begin{equation}\label{eq:problem-orig}
\partial_t u = \Delta u + f(u), \quad \alpha u + \beta \partial_n u \vert_{\partial \Omega} = b.
\end{equation}
Depending on the choice of $\alpha$ and $\beta$, it is equipped with Dirichlet, Neumann, Robin, or a mixture of these boundary conditions.

The problem of order reduction in splitting methods has the following origin. When the right-hand side of the differential equation \eqref{eq:problem-orig} is split into the two parts $\Delta u$ and $f(u)$, the nonlinearity $f(u)$ does not satisfy, in general, the homogeneous boundary conditions $\alpha u + \beta \partial_n u \vert_{\partial \Omega}=0$ with $f(u)$ in place of $u$. However, these boundary conditions are essential as they characterize the domain of the Laplacian $\Delta$, which plays a paramount role in the error analysis. The noncompliance of the nonlinearity with the homogeneous boundary conditions is the main source of order reduction.

We therefore look for a function $q$ that satisfies the boundary conditions of the nonlinearity $f(u)$, that is
\begin{equation}
\alpha q + \beta \partial_n q \vert_{\partial \Omega} = \alpha f(u) + \beta f^{\prime}(u) \partial_n u \vert_{\partial \Omega},
\end{equation}
and define the splitting
\begin{equation}
\begin{aligned}
\partial_t v &= \Delta v + q, \quad \alpha u + \beta \partial_n u \vert_{\partial \Omega} = b\\
\partial_t w &= f(w) - q
\end{aligned}
\end{equation}
with the help of this correction. Then $\partial_t w$ satisfies the desired boundary conditions.

Note that $q$ still depends on the value of $u$ and $\partial_n u$ on the boundary. We can eliminate $\partial_n u$ from the relation by using $\beta \partial_n u \vert_{\partial \Omega} = b-\alpha u \vert_{\partial \Omega}$. This then yields
\[
\alpha q + \beta \partial_n q \vert_{\partial \Omega} = \alpha f(u) + f^{\prime}(u)(b-\alpha u) \vert_{\partial \Omega}.
\]
It is instructive to consider a few special cases. For Dirichlet boundary conditions ($\alpha=1$, $\beta=0$) we obtain the familiar relation
\[ q \vert_{\partial \Omega} = f(b) \]
which, in the notation of our previous paper \cite{EO15overcoming} constitutes two steps. First, we compute an extension $z$ of the boundary data, i.e. $z\vert_{\partial \Omega}= b$, and then $q=f(z)$ which (on the boundary) yields the same result. Note that $q$ is independent of $u$ and can be precomputed (for time independent boundary data $b$) in this case.

For Neumann boundary conditions ($\alpha=0$, $\beta=1$) we obtain
\[ \partial_n q \vert_{\partial \Omega} = f^{\prime}(u)b. \]
Note that in this case $q$ depends on $u\vert_{\partial \Omega}$. However, we can substitute $u$ by the numerical solution at the beginning of the time step and still obtain a second-order approximation, see section~\ref{sec:strang}.

In the general case, the continuation $q$ will thus be constraint by
\begin{equation}\label{eq:exact-boundary}
\alpha q(t_n) + \beta \partial_n q(t_n) \vert_{\partial \Omega} = \alpha f(u_n) + f^{\prime}(u_n)(b-\alpha u_n) \vert_{\partial \Omega},
\end{equation}
where $u_n$ denotes the numerical solution at time $t_n$.

The outline of the paper is as follows. In section \ref{sec:model} we present our model problem, a semilinear parabolic equation subject to oblique boundary conditions. The numerical scheme is explained in section \ref{sec:method}. With the help of a correction $q$ that has to satisfy a certain boundary condition, we define a so-called boundary-corrected splitting. Section~\ref{sec:error-analysis} addresses the error analysis of our modified splitting schemes. The convergence properties of the Lie splitting are analyzed in section~\ref{sec:lie}, that of the Strang splitting in section~\ref{sec:strang}. The main result in this context is Theorem~\ref{thm:strang} which proves second order convergence of the modified Strang splitting under natural assumptions on the data. The order reduction, which is observed for the classic Strang splitting, is further discussed and rigorously explained in section~\ref{sec:explain}. Some representative numerical results in one and two space dimensions are given in sections~\ref{sec:1d} and~\ref{sec:2d}, respectively. Issues of efficient implementation are finally discussed in section~\ref{sec:implement}.

We conclude this section by some remarks on the employed notation. The letter $C$ always denotes a generic constant. It may take on different numeric values at different occurrences. Most functions in our paper depend on time $t$ and space $x$. Whenever convenient, however, we suppress the second variable in the notation and simply write $g(t)$ instead of $g(t,\cdot)$. Further, we denote $g_n(s) = g(t_n+s)$ for $t_n=n\tau$, where $\tau$ is the time step size of the numerical method. The very exception to this rule is the exact solution of the problem, denoted by $u$. Here, the symbol $u_n$ always denotes the numerical approximation to $u(t)$ at time $t=t_n$.

\section{Model problem}
\label{sec:model}

The purpose of this section is to present a model problem for which a rigorous error analysis of Lie and Strang splitting will be provided. On the one hand, the problem should be simple enough so that all steps in the proof can be presented in detail. On the other hand, it should be sufficiently general to contain interesting problem classes. We choose a class of semilinear parabolic problems subject to oblique boundary conditions.

Let $\Omega\subset \mathbb R^d$ be an open, bounded subset with smooth boundary $\partial \Omega$. We consider a second-order elliptic differential operator
\begin{equation}\label{eq:diff-op}
D = \sum_{i,j=1}^d d_{ij}(x)\partial_{ij} + \sum_{i=1}^d d_i(x)\partial_i + d_0(x)I
\end{equation}
with smooth coefficients. The matrix $(d_{ij}(x))$ is assumed to be symmetric and uniformly positive definite on $\Omega$. Further, we consider the first-order boundary operator%
\begin{equation}
B = \sum_{i=1}^d \beta_i(x)\partial_i + \alpha(x)I
\end{equation}
with sufficiently smooth coefficients. The operator $B$ is assumed to satisfy the uniform non tangentiality condition
\begin{equation}\label{eq:nontang}
\inf_{x\in\partial\Omega}\left|\textstyle\sum_{i=1}^d \beta_i(x) n_i(x) \right| >0.
\end{equation}
Here, $n(x)$ denotes the outer normal of $\Omega$ at $x\in \partial\Omega$. We further assume that $f$ is a sufficiently smooth real function. For more details concerning the functional analytic framework, we refer to \cite[Sect.~3.1]{lunardi95}.

Henceforth, we consider the following semilinear parabolic problem with oblique boundary conditions
\begin{subequations}\label{eq:pde2}
\begin{align}
\partial_t u &= Du + f(u)\\
Bu|_{\partial\Omega} &= b\\
u(0) &= u_0
\end{align}
\end{subequations}
as the model problem in our analysis. The nonlinearity $f$, the inhomogeneity $b$ and the initial data $u_0$ are assumed to be sufficiently smooth. Depending on the choice of coefficients in $B$, different boundary conditions are modelled. For example, the choice
$$
\beta_1 = \ldots = \beta_s=0
$$
gives a Dirichlet problem, whereas the choice
$$
\alpha = 0,\qquad \beta_i(x) = \textstyle\sum_{j=1}^d d_{ij}(x) n_j(x), \quad 1\le i\le d
$$
corresponds to a Neumann problem. Note that \eqref{eq:nontang} is satisfied in both of these cases.

\section{Description of the numerical method}
\label{sec:method}

We next describe how to carry out one splitting step in the time interval $[t_n,t_{n+1}]$ of length $\tau$. As motivated in the introduction, we choose as correction a smooth function $q$ that satisfies the boundary conditions of $f(u)$. Actually, it is sufficient to require (as we will see later)
\begin{equation}\label{eq:bd-comp}
Bq_n(0)|_{\partial\Omega} = Bf(u(t_n))|_{\partial\Omega} + \mathcal O(\tau).
\end{equation}
Since our numerical methods converge at least with order one, we can simply take
\begin{equation}\label{eq:bd-choice}
Bq_n|_{\partial\Omega} = \alpha f(u_n) + f'(u_n)\bigl( b_n -\alpha u_n \bigr)|_{\partial\Omega},
\end{equation}
as this choice then satisfies~\eqref{eq:bd-comp}. With this correction $q_n$ at hand, we consider the boundary-corrected splitting
\begin{subequations}\label{eq:split}
\begin{align}
\partial_t v_n &= Dv_n + q_n,\qquad Bv_n|_{\partial\Omega} =b_n \label{eq:split-v}\\
\partial_t w_n &= f(w_n)-q_n,\label{eq:split-w}
\end{align}
\end{subequations}
and solve it on the time interval $[t_n,t_{n+1}]$ by the standard Lie or Strang approach.

In the case of Lie splitting, for a given initial value $u_n$, we first\footnote{We note that the order of the two partial flows can also be interchanged. The resulting scheme can be analyzed in a similar way, see also Lemma~\ref{lem:lie-rev} below.} solve \eqref{eq:split-w} with initial value $w_n(0)=u_n$ to obtain $w_n(\tau)$. Then, we integrate \eqref{eq:split-v} with initial value $v_n(0)=w_n(\tau)$ and finally define the numerical solution $u_{n+1}$ at time $t_{n+1}$ as
\begin{equation}\label{eq:split-lie}
u_{n+1}=\mathcal L_\tau u_n = v_n(\tau).
\end{equation}
Henceforth, the numerical scheme \eqref{eq:split-lie} with operator $\mathcal L_\tau$ will be referred to as the \emph{modified} Lie splitting step of size $\tau$.

In the case of Strang splitting, for a given initial value $u_n$, we first solve \eqref{eq:split-v} with initial value $v_n(0)=u_n$ to obtain $v_n(\tfrac{\tau}2)$. Next, we integrate \eqref{eq:split-w} with initial value $w_n(0)=v_n(\tfrac{\tau}2)$ to obtain $w_n(\tau)$. Finally, we integrate once more \eqref{eq:split-v}, but this time with initial value $v_n(0)=w_n(\tau)$, and define the numerical solution $u_{n+1}$ at time $t_{n+1}$ as
\begin{equation}\label{eq:split-strang}
u_{n+1}=\mathcal S_\tau u_n = v_n(\tfrac{\tau}2).
\end{equation}
The numerical scheme \eqref{eq:split-strang} with operator $\mathcal S_\tau$ will be referred to as the \emph{modified} Strang splitting step of size $\tau$ henceforth.

Note that the classic Lie and Strang splitting schemes can formally also be analyzed in the above framework, simply by setting $q_n=0$. We will make use of that later on. For the choice $q_n=0$, however, condition \eqref{eq:bd-comp} is not satisfied, in general.

\section{Error analysis}
\label{sec:error-analysis}

The purpose of this section is to give a thorough error analysis of Lie and Strang splitting applied to \eqref{eq:pde2}. Our analysis will be based on the framework of analytic semigroups, see \cite{pazy83, henry81, lunardi95}. For the purpose, we first have to transform the partial differential equation \eqref{eq:pde2} to homogeneous boundary conditions. Note, however, that this transformation is only employed in the proof and will not affect our numerical methods.

Let $z$ denote a smooth function that fulfills the boundary conditions
\begin{equation}\label{eq:zbd}
Bz|_{\partial\Omega} = b,
\end{equation}
and let $\wh u = u -z$. Then $\wh u$ satisfies the abstract initial value problem
\begin{equation}\label{eq:abstract}
\partial_t \wh u = A\wh u + f(\wh u + z) + Dz - \partial_t z, \quad \wh u(0) = u(0) - z(0),
\end{equation}
where $A$ is the infinitesimal generator of an analytic semigroup on the Banach space $L^p(\Omega)$. Its domain
$$
\mathcal D(A) = \{ \psi\in W^{2,p}(\Omega) \mid B\psi|_{\partial\Omega}=0\} \subset L^p(\Omega)
$$
incorporates the homogeneous boundary conditions of the problem. For functions $\psi\in \mathcal D(A)$, the action of $A$ is defined by $A\psi = D\psi$. For details, we again refer to \cite[Sect.~3.1]{lunardi95}.

For later use we recall that analytic semigroups enjoy the parabolic smoothing property, i.e.
\begin{equation}\label{eq:para-smooth}
\Vert \ee^{t A}(-A)^\gamma \Vert \leq Ct^{-\gamma},\quad\gamma\ge 0
\end{equation}
for all $t\in(0,T]$. Without any loss of generality, we assumed here that $A$ is invertible with a bounded inverse. This can always be achieved by an appropriate scaling of $u$.

Let $\varphi_k$ denote the entire functions, recursively defined by
$$
\varphi_{k+1}(z) = \frac{\varphi_k(z)-\frac1{k!}}z,\quad k\ge 0,\qquad \varphi_0(z)=\ee^z.
$$
From the integral representation
\begin{equation}\label{eq:varphi-int}
\varphi_k(z) = \int_0^1 \! \ee^{(1-s)z}\frac{s^{k-1}}{(k-1)!}\,\dd s,\quad k\ge 1,
\end{equation}
we infer that the operators $\varphi_k(tA)$ are uniformly bounded for $0\le t \le T$. Moreover, the following relation holds
\begin{equation}\label{eq:varphi-rec}
\ee^{\tau A} - \varphi_1(\tau A) = \tau A \bigl( \varphi_1(\tau A) - \varphi_2(\tau A)\bigr),
\end{equation}
which will be employed later.

We still have to formulate an appropriate framework to include the nonlinearity $f$. For simplicity, we want to use the same norm for the analysis of \eqref{eq:split-v} and~\eqref{eq:split-w}. This, however, implies that $f$ must be considered as a smooth function on $L^p(\Omega)$, defined in a neighborhood of the exact solution with values in $L^p(\Omega)$. This seems to be quite a restrictive assumption. However, we can use the well-known fact that the norm is differentiable in $L^p(\Omega)\setminus\{0\}$ for $1<p<\infty$ and twice differentiable for $2\le p<\infty$, see \cite[Sect.~2.2]{sand85}. Therefore, $f(u)$ can be replaced in the analysis by $\chi(\|u\|) f(u)$, where $\chi$ is a suitably chosen real smooth cut-off function. Without further mention, we will follow this approach.

Still another possibility would be to consider $f:V\to L^p(\Omega)$, where $V$ denotes the fractional power space $V=\mathcal D((-A)^\gamma)$ for some $0\le \gamma < 1$. This framework is used often in the literature, for example in \cite{henry81}. For simplicity, however, we will not consider it further here.

The exact solution of the abstract problem~\eqref{eq:abstract} can be expressed by the variation-of-constants formula
$$
\wh u(t) = \ee^{\tau A}\wh u(0) + \int_0^t \ee^{(t-s) A}\Bigl(f(\wh u(s)+z(s)) + Dz(s) -\partial_t z(s)\Bigr) \dd s.
$$
This allows us to express the exact solution of \eqref{eq:pde2} at time $t_{n+1}=t_n+\tau$ in the following way:
\begin{equation}\label{eq:exact-sol}
\begin{aligned}
u(t_{n+1}) &= z_n(\tau) + \ee^{\tau A}\bigl(u(t_n)-z_n(0)\bigr) \\
&\qquad + \int_0^\tau \ee^{(\tau-s) A}\Bigl(f(u(t_n+s)) + Dz_n(s) -\partial_t z_n(s)\Bigr) \dd s,
\end{aligned}
\end{equation}
where $z_n(s)$ is a smooth function satisfying the boundary condition $Bz_n(0)|_{\partial\Omega}=b(t_n)$. For example, we can take $z_n(s)=z(t_n+s)$.

The convergence analysis is quite similar to that one carried out in our previous paper \cite{EO15overcoming} in the Dirichlet case. The main difference comes from the fact that the correction $q$ is now solution dependent, in general.

\subsection{Lie and modified Lie splitting}
\label{sec:lie}

We commence this section by studying the local error of Lie splitting applied to \eqref{eq:pde2}. For this purpose, we start the numerical solution at time $t_n$ with the initial value $\wt u_n = u(t_n)$ on the exact solution. We first carry out a substep of size $\tau$ with the vector field \eqref{eq:split-w}, which gives
\begin{equation}\label{eq:lie1}
w_n(\tau) = \wt u_n + \tau \bigl(f(\wt u_n) - q_n(0)\bigr) + \mathcal O (\tau^2),
\end{equation}
and then conclude the Lie step by integrating \eqref{eq:split-v} with initial value $v_n(0)=w_n(\tau)$ up to time $\tau$. This provides the sought-after numerical solution
\begin{equation}\label{eq:lie2}
\begin{aligned}
\mathcal L_\tau \wt u_n &= z_n(\tau) + \ee^{\tau A}\bigl(w_n(\tau)-z_n(0)\bigr) \\
&\qquad + \int_0^\tau \ee^{(\tau-s) A}\Bigl(q_n(s) + Dz_n(s) -\partial_t z_n(s)\Bigr) \dd s.
\end{aligned}
\end{equation}
The local error $\delta_{n+1} = \mathcal L_\tau \widetilde u_n - u(t_{n+1})$ is obtained by inserting \eqref{eq:lie1} into \eqref{eq:lie2} and subtracting \eqref{eq:exact-sol}. This gives the following representation of the local error
\begin{equation}\label{eq:locerr-lie}
\begin{aligned}
\delta_{n+1} &= \tau\ee^{\tau A}\bigl( f(\wt u_n) - q_n(0)\bigr)\\
&\qquad + \int_0^\tau \ee^{(\tau-s) A}\Bigl(q_n(s) - f(u(t_n+s))\Bigr) \dd s + \mathcal O(\tau^2).
\end{aligned}
\end{equation}
Expanding
$$
q_n(s) - f(u(t_n+s)) = q_n(0) - f(\wt u_n) + \mathcal O(\tau)
$$
we get
\begin{equation}\label{eq:loc-err-lie}
\begin{aligned}
\delta_{n+1} &= \tau\bigl(\ee^{\tau A}-\varphi_1(\tau A)\bigr)\bigl( q_n(0) - f(\wt u_n)\bigr) + \mathcal O(\tau^2)\\
&= \tau^2 A \bigl(\varphi_1(\tau A)-\varphi_2(\tau A)\bigr)\bigl( q_n(0) - f(\wt u_n)\bigr) + \mathcal O(\tau^2),
\end{aligned}
\end{equation}
where we used \eqref{eq:varphi-int} and \eqref{eq:varphi-rec}. Henceforth, we will employ the following assumption on the data of~\eqref{eq:pde2}.

\begin{assumption} \label{ass1}
Let the domain $\Omega$, the differential operators $D$ and $B$, and the inhomogeneity $d$ satisfy the assumptions of section~\ref{sec:model}, let $f$ be differentiable, and assume that $u_0$ is smooth and satisfies the boundary conditions.
\end{assumption}

Under these assumptions $A$ generates an analytic semigroup \cite[Sect.~3.1]{lunardi95}. Moreover, Proposition~7.1.10 in \cite{lunardi95} shows that the solution $u$ of \eqref{eq:pde2} is continuously differentiable.

\begin{theorem}[Convergence of the classic Lie splitting] \label{thm:lie}
Under Assumption~\ref{ass1}, the classic Lie splitting is convergent of order $\tau \left\vert \log \tau \right\vert$, i.e., the global error satisfies the bound
$$
\|u_n-u(t_n)\| \le C\tau (1+\left\vert \log \tau \right\vert), \qquad 0\le n\tau \le T,
$$
where the constant $C$ depends on $T$ but is independent of $\tau$ and $n$.
\end{theorem}

\begin{proof}
The global error $e_n = u_n - u(t_n)$ satisfies the recursion
\begin{equation}
e_{n+1} = \mathcal L_\tau u_n -\mathcal L_\tau \wt u_n + \delta_{n+1}
\end{equation}
which, by \eqref{eq:lie1} and \eqref{eq:lie2}, can be brought into the following form:
\begin{equation}
e_{n+1}= \ee^{\tau A}\Bigl( e_n + \tau \bigl(f(u_n) - f(\wt u_n)\bigr)\Bigr) + \delta_{n+1} + \mathcal O(\tau^2).
\end{equation}
Solving this recursion, using the local Lipschitz continuity of $f$, the very form of the local errors \eqref{eq:loc-err-lie} with $q_n=0$, and the parabolic smoothing property \eqref{eq:para-smooth}, we get
\[
\Vert e_n \Vert \leq C\|e_0\| + C\tau \sum_{k=0}^{n-1} \Vert e_k \Vert + C \tau^2 \sum_{k=1}^{n-1} \frac{1}{k \tau} + C\tau.
\]
Note that the third term on the right-hand side can by estimated by $C \tau (1+\left\vert \log \tau \right\vert)$. This, together with Gronwall's inequality and $\|e_0\|=0$ gives the desired bound.
\end{proof}

Under the slightly stronger assumption
\begin{equation}\label{eq:bd-relax}
q_n(0) - f(u(t_n)) \in\mathcal D((-A)^\gamma) \qquad \text{for some $0 <\gamma \le 1$},
\end{equation}
the $\log \tau$ term in the above theorem can be omitted. This follows at once from the representation of the local errors
$$
\delta_{n+1} = -\tau^2(-A)^{1-\gamma} \bigl(\varphi_1(\tau A)-\varphi_2(\tau A)\bigr)(-A)^\gamma\bigl( q_n(0) - f(u(t_n))\bigr) + \mathcal O(\tau^2).
$$
For the oblique boundary conditions considered in this paper, we can choose any $\gamma <\tfrac1{2p}$, see~\cite{grisvard1967}.

In particular, for the modified Lie splitting (which is equivalent to setting $\gamma = 1$ in \eqref{eq:bd-relax}) we get the following result.
\begin{theorem}[Convergence of the modified Lie splitting] \label{thm:lie-corr}
Under Assumption~\ref{ass1}, the modified Lie splitting with $q_n$ satisfying \eqref{eq:bd-choice} is convergent of order one.
\end{theorem}

\subsection{Strang and modified Strang splitting}
\label{sec:strang}

In order to study the convergence properties of Strang splitting and its modifications, we again first analyze its local error when applied to \eqref{eq:pde2}. For this purpose, we consider one step of the numerical solution, starting at time $t_n$ with the initial value $\wt u_n = u(t_n)$ on the exact solution. The first half step of the modified Strang splitting (see \eqref{eq:split-v}) is then given by
\begin{equation}\label{eq:strang1}
\begin{aligned}
v_n(\tfrac{\tau}2) &= z_n(\tfrac{\tau}2) + \ee^{\frac{\tau}2 A}\bigl(\wt u_n - z_n(0)\bigr) \\
&\qquad + \int_0^{\frac{\tau}2} \ee^{(\frac{\tau}2-s) A}\Bigl(q_n(s) + Dz_n(s) -\partial_t z_n(s)\Bigr) \dd s,
\end{aligned}
\end{equation}
where $z_n$ denotes a sufficiently smooth function that satisfies the boundary condition $Bz_n(0)|_{\partial\Omega}=b_n(0)$. We use here the same function $z_n$ as in \eqref{eq:exact-sol}. The subsequent full step with the vector field \eqref{eq:split-w} then has the representation
\begin{equation}\label{eq:strang2}
\begin{aligned}
w_n(s) &= v_n(\tfrac{\tau}2) +  t \bigl(f(w_n(\tfrac{s}2)) - q_n(\tfrac{s}2)\bigr)+ \mathcal O (s^3),\qquad 0\le s \le \tau.
\end{aligned}
\end{equation}
Note that we have employed a symmetric expansion based on the midpoint of the interval $[0,s]$. Therefore, the $\mathcal O (s^2)$ term does not show up in the expansion. Carrying out again one half step of \eqref{eq:split-v} with starting value $w_n(\tau)$ provides one step of Strang splitting
\begin{equation}\label{eq:strang3}
\begin{aligned}
\mathcal S_h \wt u_n &= z_n(\tau) + \ee^{\frac{\tau}2 A}\bigl(w_n(\tau)-z_n(\tfrac{\tau}2)\bigr)\\
&\qquad + \int_0^{\frac{\tau}2} \ee^{(\frac{\tau}2-s) A}\Bigl(q_n(\tfrac{\tau}2+s) + Dz_n(\tfrac{\tau}2+s) -\partial_t z_n(\tfrac{\tau}2+s)\Bigr) \dd s.
\end{aligned}
\end{equation}
Inserting \eqref{eq:strang2} and \eqref{eq:strang1} into \eqref{eq:strang3} finally gives the following representation of the numerical solution
\begin{equation}\label{eq:st-local}
\begin{aligned}
\mathcal S_h \wt u_n &= z_n(\tau) + \ee^{\tau A}\bigl(\wt u_n-z_n(0)\bigr) + \tau\ee^{\tau A}\bigl(f(w_n(\tfrac{\tau}2))-q_n(\tfrac{\tau}2)\bigr)\\
&\qquad + \int_0^{\tau} \ee^{(\tau-s) A}\Bigl(q_n(s) + Dz_n(s) - \partial_t z_n(s))\Bigr) \dd s + \mathcal O(\tau^3).
\end{aligned}
\end{equation}
The local error is given by $\delta_{n+1} = \mathcal S_h \wt u_n  - u(t_{n+1})$. Taking the difference of the numerical \eqref{eq:st-local} and the exact solution \eqref{eq:exact-sol}, we get
\begin{equation}\label{eq:locerr-strang}
\begin{aligned}
\delta_{n+1} &= \tau\ee^{\frac{\tau}2 A}\bigl(f(w_n(\tfrac{\tau}2)) - f(u(t_n+\tfrac{\tau}2))\bigr)+ \tau\ee^{\frac{\tau}2 A}\bigl(f(u(t_n+\tfrac{\tau}2))- q_n(\tfrac{\tau}2)\bigr)\\
&\qquad + \int_0^{\tau} \ee^{(\tau-s) A}\Bigl(q_n(s) - f(u(t_n+s))\Bigr) \dd s + \mathcal O(\tau^3)\\
&=\tcone+\tctwo+\tcthree+ \mathcal O(\tau^3).
\end{aligned}
\end{equation}
This error representation is composed of three terms and a remainder. In order to bound the first term in \eqref{eq:locerr-strang}, we use the Lipschitz continuity of $f$ and the bound
\begin{equation}\label{eq:lie-rev}
\|w_n(\tfrac{\tau}2) - u(t_n+\tfrac{\tau}2) \| \le C \tau^2,
\end{equation}
which holds uniformly in $n$ and $\tau$ on compact time intervals (see Lemma~\ref{lem:lie-rev} below). This shows that \tcone\ is of size $\mathcal O(\tau^3)$.

The remaining terms \tctwo\ and \tcthree\ are treated together. Employing the Peano kernel representation of the error of the midpoint rule
\begin{equation}\label{eq:midpoint}
\int_0^\tau \psi(s)\,\dd s = \tau\psi(\tfrac{\tau}2) + \int_0^{\frac{\tau}2} \frac{s^2}2\psi''(s)\,\dd s + \int_{\frac{\tau}2}^\tau \frac{(\tau-s)^2}2\psi''(s)\,\dd s
\end{equation}
with
$$
\psi(s) = \ee^{(\tau-s) A}\wh \psi(s), \qquad \wh\psi(s)= q_n(s) - f(u(t_n+s))
$$
shows that we just have to estimate the two integral remainder terms on the right-hand side of \eqref{eq:midpoint}. The term $\psi''(s)$ consists of three different terms
$$
\psi''(s) = A\ee^{(\tau-s) A}A\wh\psi(s) - 2A\ee^{(\tau-s) A}\wh\psi'(s) + \ee^{(\tau-s) A}\wh\psi''(s).
$$
The leftmost $A$ in the first two terms can be compensated by parabolic smoothing in exactly the same way as for Lie splitting (see also the proof of Theorem~\ref{thm:strang} below). Therefore, the only term that requires attention is
\begin{equation}\label{eq:fractional-a}
\chi(s) = \ee^{(\tau-s) A}A\wh\psi(s) = \ee^{(\tau-s) A}A\bigl(\wh q_n(0) - f(\wt u_n)\bigr) + \rho_n.
\end{equation}
Due to the parabolic smoothing property, the remainder $\rho_n$ satisfies the bound
$$
\|\rho_n\| \le C\tau(\tau -s)^{-1}
$$
uniformly in $n$. This together with \eqref{eq:bd-comp} shows that $\chi$ is uniformly bounded, and proves the following representation of the local error
\begin{equation}\label{eq:st-corr-loc-err}
\delta_{n+1} =  A\wh\delta_{n+1} + \mathcal O(\tau^3), \qquad \wh\delta_{n+1} = \mathcal O(\tau^3).
\end{equation}

We will employ the following assumption on the data of~\eqref{eq:pde2}.

\begin{assumption}\label{ass2}
Let the domain $\Omega$, the differential operators $D$ and $B$, and the inhomogeneity $d$ satisfy the assumptions of section~\ref{sec:model}, let $f$ be twice differentiable, and assume that $u_0$ and $Du_0$ are smooth and satisfy the boundary conditions.
\end{assumption}

We are now in the position to state the convergence result for the modified Strang splitting.

\begin{theorem}[Convergence of the modified Strang splitting] \label{thm:strang}
Let $q_n$ satisfy \eqref{eq:bd-choice}. Then, under Assumption~\ref{ass2} the modified Strang splitting scheme is convergent of order $\tau^2 \left\vert \log \tau \right\vert$, i.e., the global error satisfies the bound
$$
\|u_n-u(t_n)\| \le C\tau^2 (1+\left\vert \log \tau \right\vert), \qquad 0\le n\tau \le T,
$$
where the constant $C$ depends on $T$ but is independent of $\tau$ and $n$.
\end{theorem}

Under the assumption
\begin{equation*}
q_n(0) - f(u(t_n)) \in\mathcal D((-A)^{1+\gamma}) \qquad \text{for some $0 <\gamma \le 1$},
\end{equation*}
which is slightly stronger than \eqref{eq:bd-choice}, the $\log \tau$ term in the above theorem can be omitted. This follows in the same way as the corresponding result for Lie splitting, discussed above. For the oblique boundary conditions considered in this paper, we can choose again any $\gamma <\tfrac1{2p}$, see \cite{grisvard1967}.

\def\myproofname{Theorem~\ref{thm:strang}}
\begin{proofof}
The global error $e_n = u_n - u(t_n)$ satisfies the recursion
\begin{equation}
e_{n+1} = \mathcal S_\tau u_n -\mathcal S_\tau \wt u_n + \delta_{n+1}
\end{equation}
which, by \eqref{eq:st-local} and \eqref{eq:strang2}, can be brought to the following form
\begin{equation}
e_{n+1}= \ee^{\tau A}\bigl( e_n + \tau E(u_n,\wt u_n)\bigr) + \delta_{n+1} + \mathcal O(\tau^2),
\end{equation}
where $E$ satisfies the bound $\|E(u_n,\wt u_n)\|\le C \|e_n\|$ uniformly in $n$. Using the representation \eqref{eq:st-corr-loc-err} of the local errors $\delta_{n+1}$, we can proceed from here on literally as in the proof of Theorem~\ref{thm:lie}.
\end{proofof}

The following Lemma was used for the proof of Theorem~\ref{thm:strang}. It also shows that modified Lie splitting, carried out with the flows interchanged, has the same order of convergence.
\begin{lemma}\label{lem:lie-rev}
Let $q_n$ satisfy \eqref{eq:bd-choice}. Then, under the assumptions of Theorem~\ref{thm:lie}, the following bound holds
\begin{equation}\label{eq:lie-rev-lem}
\left\|w_n(\tfrac{\tau}2) - u(t_n+\tfrac{\tau}2) \right\| \le C \tau^2
\end{equation}
with a constant $C$ that can be chosen independently of $n$ and $\tau$ on compact time intervals $0\le t_n=n\tau \le T$.
\end{lemma}

\begin{proof}
From~\eqref{eq:strang2} with $t=\tfrac{\tau}2$ we infer that
$$
w_n(\tfrac{\tau}2) = v_n(\tfrac{\tau}2) + \tfrac{\tau}2\bigl(f(w_n(0))-q_n(0)\bigr) + \mathcal O(\tau^2)
$$
with $v_n(\tfrac{\tau}2)$ given by~\eqref{eq:strang1}.
Using again the notation $\wh\psi(s) = q_n(s) - f(u(t_n+s))$, we have
\begin{equation*}
\begin{aligned}
w_n(\tfrac{\tau}2) - u(t_n+\tfrac{\tau}2) &= \tfrac{\tau}2\bigl(f(w_n(0))-q_n(0)\bigr) + \int_0^{\frac{\tau}2}\ee^{(\frac{\tau}2-s)A}\wh\psi(s)\,\dd s\\
&=\tfrac{\tau}2\bigl(f(w_n(0))-f(\wt u_n)\bigr) + \int_0^{\frac{\tau}2}\!\!\int_0^s \ee^{(\frac{\tau}2-\xi)A}(A\wh\psi(\xi)+\wh\psi'(\xi))\,\dd\xi\,\dd s.
\end{aligned}
\end{equation*}
Since $q_n$ satisfies the boundary conditions of $f(\wt u_n)$, the above integrand fulfills the bound
\begin{equation}\label{eq:fractional-b}
\|\ee^{(\frac{\tau}2-\xi)A}(A\wh\psi(\xi)+\wh\psi'(\xi))\| \le C
\end{equation}
and the double integral is appropriately bounded. Further, since $f$ is locally Lipschitz continuous, it remains to estimate the difference
\begin{align*}
w_n(0)-\wt u_n &= v_n(\tfrac{\tau}2)-\wt u_n\\
&= z_n(\tfrac{\tau}2) + \ee^{\frac{\tau}2 A}\bigl( \wt u_n - z_n(0)\bigr) -\wt u_n \\
&= \tfrac{\tau}2 A  \,\varphi_1(\tfrac{\tau}2 A) \bigl( z_n(0) - \wt u_n \bigr) +\mathcal O(\tau).
\end{align*}
But both, $z_n(0)$ and $\wt u_n$ satisfy the boundary conditions. Therefore, their difference $z_n(0)-\wt u_n$ lies in the domain of $A$. This gives at once the required bound~\eqref{eq:lie-rev-lem}.
\end{proof}

\subsection{Explanation of the encountered fractional orders of convergence}
\label{sec:explain}

\begin{table}[bt]
\caption{Expected orders of convergence for classic Strang splitting in various norms.}
\label{tab:orders}
\begin{center}
\renewcommand{\arraystretch}{1.15}
\setlength{\tabcolsep}{2mm}
\begin{tabular}{c|ccc}
boundary type & $L^1$ & $L^2$ & $L^{\infty}$ \\
\hline
$\beta_1 = \ldots = \beta_d=0$ & 1.50 & 1.25 & 1.00\\
$\exists j \text{ with } \beta_j \ne 0$ & 2.00 & 1.75 & 1.50\\
\end{tabular}
\end{center}
\end{table}

We are now in the position to explain the order reduction which is encountered in our experiments with the classic Strang scheme. Under the relaxed assumption \eqref{eq:bd-relax}, we get instead of \eqref{eq:fractional-b} the weaker bound
\begin{equation}
\|\ee^{(\frac{\tau}2-\xi)A}(A\wh\psi(\xi)+\psi'(\xi))\| \le C(\tfrac{\tau}2-\xi)^{\gamma-1}.
\end{equation}
This is again achieved by employing the parabolic smoothing property~\eqref{eq:para-smooth}. Consequently, we get
$$
\int_0^{\frac{\tau}2}\!\!\int_0^s \ee^{(\frac{\tau}2-\xi)A}(A\wh\psi(\xi)+\psi'(\xi))\,\dd\xi\,\dd s = C \tau^{\gamma+1},
$$
which shows that the assertion of Lemma~\ref{lem:lie-rev} still holds under the assumption \eqref{eq:bd-relax} with the weaker bound
$$
\left\|w_n(\tfrac{\tau}2) - u(t_n+\tfrac{\tau}2) \right\| \le C \tau^{\gamma+1}.
$$
Next, we rewrite~\eqref{eq:fractional-a} as
\begin{equation}
\chi(s) = \ee^{(\tau-s) A}(-A)^{1-\gamma} \cdot (-A)^\gamma\bigl(\wh q_n(0) - f(\wt u_n)\bigr) + \rho_n
\end{equation}
and use again~\eqref{eq:para-smooth}. This gives the bound
$$
\|\chi(s)\| \le C(\tau -s)^{\gamma -1} + C\tau(\tau-s)^{-1}.
$$
Inserting this into \eqref{eq:midpoint} shows that
\begin{equation}
\delta_{n+1} =  A \wh \delta_{n+1} + \mathcal O(\tau^3),\qquad  \wh\delta_{n+1} = \mathcal O(\tau^{\gamma +2}).
\end{equation}
The classic Strang splitting under assumption \eqref{eq:bd-relax} thus converges with order $1+\gamma$. The value of $\gamma$ depends on the type of boundary condition and the chosen norm in which the error is measured. Table~\ref{tab:orders} gives some typical values, see \cite{fujiwara1967, grisvard1967}.

\section{Numerical results in one space dimension}
\label{sec:1d}

In this section we will present a number of numerical results for the diffusion-reaction problem \eqref{eq:problem-orig} with $f(u)=u^2$ on $\Omega=[0,1]$. The Laplacian is discretized by classical centered second-order finite differences, and we use the CVODE library \cite{cvode} to compute the partial flows given by \eqref{eq:split}. In all the simulations we will refer to the classical splitting approach by Lie and Strang, respectively, while we refer to the schemes introduced in section~\ref{sec:method} by Lie (modified) and Strang (modified), respectively.

\begin{example}[Neumann boundary conditions]
For this problem we prescribe inhomogeneous Neumann boundary conditions ($\alpha=1$, $\beta=\beta_1=0$) with $b_0=0$ and $b_1=1$. From the boundary conditions an admissible correction in accordance with condition \eqref{eq:bd-choice} is easily determined. It does not depend on time and is given by $q_n(s,x)=x^2 u_n(1)$. As initial condition we have chosen $u_0(x)=-2/\pi \cos(\tfrac{1}{2}\pi x)$. The numerical results are shown in Table \ref{fig:tab-1d-neumann}. We observe the expected reduction to order $1.5$ (infinity norm) and order $1.75$ (discrete $L^2$ norm) for the classic Strang splitting scheme. On the other hand, the modified Strang splitting scheme is seen to be second order convergent. In addition, the error constants are improved by approximately a factor of two in case of the modified Lie splitting.
\begin{table}
\caption{Diffusion-reaction equation with inhomogeneous Neumann boundary conditions and $500$ grid points. The error in the discrete infinity and $L^2$ norm, respectively, is computed at $t=0.5$ by comparing the numerical solution to a reference solution computed with the modified Strang splitting using $\tau=1\cdot10^{-4}$.
\label{fig:tab-1d-neumann}}
\centering\begin{tabular}{rrrrrrrrr}
& \multicolumn{2}{c}{Lie} && \multicolumn{2}{c}{Strang} && \multicolumn{2}{c}{Strang} \\
\cline{2-3} \cline{5-6}  \cline{8-9}
\multicolumn{1}{l}{\msp step size} & \multicolumn{1}{l}{$l^{\infty}$ error} & \multicolumn{1}{l}{order} & &
\multicolumn{1}{l}{$l^{\infty}$ error} & \multicolumn{1}{l}{order} & &
\multicolumn{1}{l}{$l^{2}$ error} & \multicolumn{1}{l}{order}\\
\hline
\msp
3.125e-02  & 2.951e-01  & --          & & 1.806e-01  & --     & & 8.733e-02  & --          \\
1.562e-02  & 1.741e-01  & 0.76       & & 2.211e-04  & 9.67   & & 2.130e-05  & 12.00       \\
7.812e-03  & 1.237e-03  & 7.14       & & 7.684e-05  & 1.52   & & 6.364e-06  & 1.74        \\
3.906e-03  & 6.315e-04  & 0.97       & & 2.638e-05  & 1.54   & & 1.895e-06  & 1.75        \\
1.953e-03  & 3.202e-04  & 0.98       & & 8.897e-06  & 1.57   & & 5.612e-07  & 1.76        \\
\end{tabular}

\vspace{2mm}
\begin{tabular}{rrrrrrrrr}
& \multicolumn{2}{c}{Lie (modified)} && \multicolumn{2}{c}{Strang (modified)} && \multicolumn{2}{c}{Strang (modified)} \\
\cline{2-3} \cline{5-6}  \cline{8-9}
\multicolumn{1}{l}{\msp step size} & \multicolumn{1}{l}{$l^{\infty}$ error} & \multicolumn{1}{l}{order} & &
\multicolumn{1}{l}{$l^{\infty}$ error} & \multicolumn{1}{l}{order} & &
\multicolumn{1}{l}{$l^{2}$ error} & \multicolumn{1}{l}{order}\\
\hline
\msp
3.125e-02  & 1.381e-01  & --          & & 8.752e-02  & --      & & 5.471e-02  & --         \\
1.562e-02  & 8.962e-02  & 0.62       & & 1.495e-05  & 12.51   & & 3.931e-06  & 13.76      \\
7.812e-03  & 7.543e-04  & 6.89       & & 3.868e-06  & 1.95    & & 9.773e-07  & 2.01       \\
3.906e-03  & 3.788e-04  & 0.99       & & 1.002e-06  & 1.95    & & 2.428e-07  & 2.01       \\
1.953e-03  & 1.898e-04  & 1.00       & & 2.603e-07  & 1.94    & & 6.023e-08  & 2.01       \\
\end{tabular}

\end{table}
\end{example}

\begin{example}[Mixed boundary conditions]
For this problem we prescribe Dirichlet boundary conditions on the left boundary (with $b_0=1$) and Neumann on the right boundary (with $b_1=1$). From these boundary conditions the 	correction is easily determined to satisfy $q_n(s,x)=1+2xu_n(1)$. As initial condition we have chosen $u_0(x)=1+\tfrac{2}{\pi}-\tfrac{2}{\pi} \cos(\tfrac{1}{2}\pi x)$. The numerical results are shown in Table \ref{table-1d-mixed}. We observe the expected reduction to order approximately $1$ (infinity norm) and order $1.25$ (discrete $L^2$ norm) for the classic Strang splitting scheme. On the other hand, the modified Strang splitting scheme is seen to be second order convergent with greatly improved accuracy even for large time step sizes. In addition, the error constants are improved by a factor of more than three in case of the modified Lie splitting.
\begin{table}
\caption{Diffusion-reaction equation with mixed Dirichlet/Neumann boundary conditions and $500$ grid points. The error in the discrete infinity and $L^2$ norm, respectively, is computed at $t=0.2$ by comparing the numerical solution to a reference solution computed with the modified Strang splitting using $\tau=5\cdot10^{-5}$.
\label{table-1d-mixed}}
\centering\begin{tabular}{rrrrrrrrr}
& \multicolumn{2}{c}{Lie} && \multicolumn{2}{c}{Strang} && \multicolumn{2}{c}{Strang}  \\
\cline{2-3} \cline{5-6} \cline{8-9}
\multicolumn{1}{l}{\msp step size} & \multicolumn{1}{l}{$l^{\infty}$ error} & \multicolumn{1}{l}{order} & &
\multicolumn{1}{l}{$l^{\infty}$ error} & \multicolumn{1}{l}{order} 
& & \multicolumn{1}{l}{$l^{2}$ error} & \multicolumn{1}{l}{order} \\
\hline\msp
1.250e-02  & 1.185e+00  & --          & & 5.718e-03  & --      & & 5.815e-04  & --         \\
6.250e-03  & 2.745e-03  & 8.75       & & 2.736e-03  & 1.06    & & 2.379e-04  & 1.29       \\
3.125e-03  & 1.430e-03  & 0.94       & & 1.288e-03  & 1.09    & & 9.652e-05  & 1.30       \\
1.563e-03  & 7.356e-04  & 0.96       & & 5.904e-04  & 1.13    & & 3.855e-05  & 1.32       \\
7.813e-04  & 3.750e-04  & 0.97       & & 2.596e-04  & 1.19    & & 1.499e-05  & 1.36       \\

\end{tabular}

\vspace{2mm}
\begin{tabular}{rrrrrrrrr}
& \multicolumn{2}{c}{Lie (modified)} && \multicolumn{2}{c}{Strang (modified)} && \multicolumn{2}{c}{Strang (modified)} \\
\cline{2-3} \cline{5-6} \cline{8-9}

\multicolumn{1}{l}{\msp step size} & \multicolumn{1}{l}{$l^{\infty}$ error} & \multicolumn{1}{l}{order} & &
\multicolumn{1}{l}{$l^{\infty}$ error} & \multicolumn{1}{l}{order} & &
\multicolumn{1}{l}{$l^{2}$ error} & \multicolumn{1}{l}{order} \\
\hline\msp
1.250e-02  & 3.618e-01  & --          & & 8.222e-05  & --       & & 2.567e-05  & --        \\
6.250e-03  & 9.110e-03  & 5.31       & & 2.087e-05  & 1.98     & & 6.426e-06  & 2.00      \\
3.125e-03  & 4.579e-03  & 0.99       & & 5.292e-06  & 1.98     & & 1.609e-06  & 2.00      \\
1.563e-03  & 2.295e-03  & 1.00       & & 1.341e-06  & 1.98     & & 4.031e-07  & 2.00      \\
7.813e-04  & 1.149e-03  & 1.00       & & 3.395e-07  & 1.98     & & 1.009e-07  & 2.00      \\

\end{tabular}

\end{table}
\end{example}

\begin{example}[Robin boundary condition]
For this problem we prescribe Robin boundary conditions ($\alpha=\beta=\beta_1=1$) and $b_0=0$, $b_1=1+2/\pi$. The correction, determined accordance with \eqref{eq:bd-choice}, satisfies $q_n(s,x)=\gamma_0 + \gamma_1 x$, where
$$
\gamma_0 = -2(1+\tfrac2{\pi})u_n(1)+u^2_n(1)-2u^2_n(0), \quad \gamma_1 = 2(1+\tfrac2{\pi})u_n(1)+u^2_n(0)-u^2_n(1).
$$
As initial condition we have chosen $u_0(x)=-2/\pi \cos(\tfrac{1}{2}\pi x)+2/\pi$. The numerical results are shown in Table \ref{table-1d-robin}. Since Robin boundary conditions behave similarly as Neumann boundary conditions we expect a reduction to order $1.5$ (infinity norm) and order $1.75$ (discrete $L^2$ norm), respectively. This is exactly what we observe for the classic Strang splitting. On the other hand, the modified Strang splitting is second order accurate. In addition, the error constants are improved by approximately a factor of two in case of the modified Lie splitting.
\begin{table}
\caption{Diffusion-reaction equation with Robin boundary conditions and $500$ grid points. The error in the discrete infinity and $L^2$ norm, respectively, is computed at $t=0.25$ by comparing the numerical solution to a reference solution computed with the modified Strang splitting using $\tau=5\cdot10^{-5}$.
\label{table-1d-robin}}
\centering\begin{tabular}{rrrrrrrrr}
& \multicolumn{2}{c}{Lie} && \multicolumn{2}{c}{Strang} && \multicolumn{2}{c}{Strang}  \\
\cline{2-3} \cline{5-6} \cline{8-9}
\multicolumn{1}{l}{\msp step size} & \multicolumn{1}{l}{$l^{\infty}$ error} & \multicolumn{1}{l}{order} & &
\multicolumn{1}{l}{$l^{\infty}$ error} & \multicolumn{1}{l}{order} & &
\multicolumn{1}{l}{$l^{2}$ error} & \multicolumn{1}{l}{order} \\
\hline\msp
1.562e-02  & 5.618e-02  & --         & & 2.388e-04  & --     & & 2.294e-05  & --          \\
7.812e-03  & 9.115e-04  & 5.95       & & 8.365e-05  & 1.51    & & 6.913e-06  & 1.73        \\
3.906e-03  & 4.690e-04  & 0.96       & & 2.890e-05  & 1.53    & & 2.072e-06  & 1.74        \\
1.953e-03  & 2.393e-04  & 0.97       & & 9.815e-06  & 1.56    & & 6.174e-07  & 1.75        \\
9.766e-04  & 1.213e-04  & 0.98       & & 3.254e-06  & 1.59    & & 1.826e-07  & 1.76        \\
\end{tabular}

\vspace{2mm}
\begin{tabular}{rrrrrrrrr}
& \multicolumn{2}{c}{Lie (modified)} && \multicolumn{2}{c}{Strang (modified)} && \multicolumn{2}{c}{Strang (modified)} \\
\cline{2-3} \cline{5-6} \cline{8-9}
\multicolumn{1}{l}{\msp step size} & \multicolumn{1}{l}{$l^{\infty}$ error} & \multicolumn{1}{l}{order} & &
\multicolumn{1}{l}{$l^{\infty}$ error} & \multicolumn{1}{l}{order} & &
\multicolumn{1}{l}{$l^{2}$ error} & \multicolumn{1}{l}{order} \\
\hline\msp
1.562e-02  & 3.232e-01  & --          & & 7.388e-06  & --      & & 4.741e-06  & --        \\
7.812e-03  & 4.716e-04  & 9.42       & & 1.865e-06  & 1.99    & & 1.200e-06  & 1.98      \\
3.906e-03  & 2.370e-04  & 0.99       & & 4.711e-07  & 1.99    & & 3.040e-07  & 1.98      \\
1.953e-03  & 1.188e-04  & 1.00       & & 1.192e-07  & 1.98    & & 7.717e-08  & 1.98      \\
9.766e-04  & 5.947e-05  & 1.00       & & 3.024e-08  & 1.98    & & 1.969e-08  & 1.97      \\
\end{tabular}

\end{table}
\end{example}

\section{Numerical results in two space dimensions}
\label{sec:2d}

\begin{table}
\caption{Two-dimensional diffusion-reaction equation with mixed Dirichlet/Neumann boundary conditions and $10^4$ quadrilateral finite elements. The error in the discrete infinity norm is computed at $t=0.1$ by comparing the numerical solution to a reference solution computed with the modified Strang splitting using $\tau=1.5\cdot10^{-3}$
\label{table-2d-mixed}}
\centering\begin{tabular}{rrrrrrrrr}
& \multicolumn{2}{c}{Strang} && \multicolumn{2}{c}{Strang (Dirichlet)} && \multicolumn{2}{c}{Strang (modified)} \\
\cline{2-3} \cline{5-6} \cline{8-9}
\multicolumn{1}{l}{\msp step size} & \multicolumn{1}{l}{$l^{\infty}$ error} & \multicolumn{1}{l}{order} & &
\multicolumn{1}{l}{$l^{\infty}$ error} & \multicolumn{1}{l}{order} & &
\multicolumn{1}{l}{$l^{\infty}$ error} & \multicolumn{1}{l}{order} \\
\hline\msp
0.1    & 6.493e-01& --     & & 1.203e-01& --      & & 1.024e-01& --\\
0.05   & 2.906e-01& 1.16& &   4.295e-02& 1.49 & &   2.568e-02& 2.00\\
0.025  & 1.345e-01& 1.11& &   1.639e-02& 1.39 & &   5.445e-03& 2.24\\
0.0125 & 6.147e-02& 1.13& &   5.455e-03& 1.59 & &   1.346e-03& 2.02\\
\end{tabular}

%
%
%

\end{table}

In this section we will consider the diffusion-reaction problem \eqref{eq:problem-orig} with $f(u)=u^2$ on $\Omega=[0,1]\times[0,1]$. The initial value is given by
$$
u_0(x)= 3+\mathrm{e}^{-5(y-\frac{1}{2})^2} \cos(2 \pi (x+y)).
$$
We impose Dirichlet boundary conditions on the top and bottom, and Neumann boundary conditions on the left and right boundary. The value of the boundary data $b$ is determined in such a manner that the initial condition $u_0$ is consistent with the prescribed boundary data. The correction $q_n$ is computed by solving Laplace's equation with these boundary conditions. The space discretization is performed with the libMesh software package \cite{libMeshPaper} and quadrilateral finite elements of second order are used.

The numerical results are shown in Table \ref{table-2d-mixed}. We observe order reduction to approximately order $1$ (in the infinity norm) for the classic Strang splitting. On the other hand, second order accuracy is observed for the modified Strang splitting. Let us also emphasize that even for large time step sizes (i.e.~a low tolerance) the modified Strang splitting scheme is more accurate by approximately one order of magnitude compared to the classic Strang splitting scheme.

Furthermore, we have displayed in Table~\ref{table-2d-mixed} the result of only applying the correction developed in \cite{EO15overcoming}. In this case the mixed nature of the boundary condition is ignored and the whole boundary is treated as if Dirichlet boundary conditions had been imposed. In this case we observe an order reduction to approximately $1.5$. Thus, some improvement can be obtained with this simplified approach. However, performing the full correction (as described in this paper) clearly offers superior accuracy.

\section{Implementation}
\label{sec:implement}

\begin{table}
\caption{Diffusion-reaction equation with Dirichlet boundary conditions and $500$ grid points. The error in the discrete infinity norm is computed at $t=0.1$ with a step size $\tau=1.25\cdot10^{-3}$ by comparing the numerical solution to a reference solution computed with the modified Strang splitting using $\tau=5\cdot10^{-5}$.
\label{table-correction}}
\centering\begin{tabular}{llr}
	method & correction & $l^{\infty}$ error \\
	\hline\msp
	Lie & none & 2.31e-03 \\
	Lie (mod.) & harmonic    & 1.20e-03 \\
	Lie (mod.) & $\sin \pi x$ & 2.43e-03 \\
	Lie (mod.) & $\sin 10 \pi x $ & 3.21e-03
\end{tabular} \hfill
\begin{tabular}{llr}
	method & correction & $l^{\infty}$ error \\
	\hline\msp
	Strang & none & 1.84e-03 \\
	Strang (mod.) & harmonic & 1.20e-06 \\
	Strang (mod.) &$\sin \pi x$ & 1.27e-06 \\
	Strang (mod.) &$\sin 10 \pi x$ & 7.02e-05
\end{tabular}

\end{table}

For oblique boundary conditions, an efficient implementation of the modified splitting schemes is more difficult to achieve than for the time independent Dirichlet boundary conditions that were the focus of \cite{EO15overcoming}. The reason for this difficulty is that the correction $q_n$ now has to be updated once at each time step and thus can not be precomputed. Note that this problem is shared by time dependent Dirichlet boundary conditions.

Therefore, in order to provide an efficient implementation of the numerical scheme described in this paper, we have to devise an efficient way to compute $q_n$. So far, this has been done exclusively by solving the corresponding elliptic problem. This is a seductive approach as it eliminates the term $D z$ from \eqref{eq:abstract}. However, as the convergence analysis carried out in section \ref{sec:error-analysis} shows, this is not the only possible choice.

Let us investigate this issue further. Assume that we have a correction $q_n$ at our disposal which satisfies the boundary conditions but is not harmonic (i.e., it does not solve the elliptic problem). From the theoretical analysis we would still expect the same order of convergence given that this correction is sufficiently smooth. Using numerical simulations we investigate the dependence of the accuracy on the correction used. In the numerical results presented in Table \ref{table-correction} we use Dirichlet boundary conditions with $b_0=1$ and $b_1=2$ and a harmonic correction (i.e.~$q_n=1+x$), a smooth correction ($q_n=1+x+\sin \pi x$) and an oscillatory correction ($q_n=1+x+\sin 10 \pi x$), respectively. We observe that the smooth correction gives results that are comparable to the harmonic correction, while the oscillatory correction results in a significant loss of accuracy. These results seem fairly robust both for Dirichlet as well as the more complicated boundary conditions considered in this paper.

We thus conclude that we can employ any correction as long as it satisfies the desired boundary conditions and is smooth enough. Therefore, we propose Algorithm~\ref{alg:qn} to compute $q_n$ (which is done at the beginning of each time step). The run time of this procedure is usually negligible compared to the effort required to solve the linear systems of equations in the parabolic substep of the splitting algorithm. Thus, for the vast majority of applications the run time of the modified Strang splitting and the classic Strang splitting are comparable and we are thus justified in comparing the accuracy of these schemes with an equal time step size.

\begin{algorithm}[t]
\caption{\quad Numerical algorithm for computing $q_n$\label{alg:qn}}
\begin{enumerate}
\item Using $u_n$ we initialize a function $\wh q$ that satisfies the boundary conditions given in \eqref{eq:bd-choice}. Note that the function $\wh q$ is not assumed to be differentiable (or even continuous). This can be accomplished in $\mathcal{O}(N)$ operations, where $N$ is the number of degrees of freedom, by performing a linear interpolation of the boundary data.
\item Apply a small number of weighted Jacobi iterations to $\wh q$ in order to obtain $q_n$. The weighted Jacobi iteration is chosen here as its purpose (as in a multigrid scheme) is to damp out the high frequencies and thus to produce a smoother solution. The slow convergence of the weighted Jacobi iteration is not an issue here since we only use it as a smoother and not to compute a solution of the elliptic problem.
\end{enumerate}
\end{algorithm}

\section{Conclusion}

We have introduced a modification of the Strang splitting scheme that succeeds in overcoming the order reduction observed for inhomogeneous Neumann, Robin, and mixed Dirichlet/Neumann boundary conditions. The present results are thus a generalization of our previous work \cite{EO15overcoming} where similar results have been obtained for Dirichlet boundary conditions. Thus, we have constructed a Strang splitting method that is provably convergent of order two even in the presence of non-trivial boundary conditions. In addition, a method to compute this modification efficiently is proposed which implies (together with the numerical results obtained in this paper) that significant gains in computational efficiency can be expected by employing this method to diffusion-reaction problems.


\bibliographystyle{plain}
\bibliography{papers}

\end{document}